\documentclass[11pt]{article}
\usepackage{amssymb, amsmath}
\usepackage[latin1]{inputenc}
\usepackage{amsmath}
\usepackage{amssymb}
\usepackage{amsfonts}
\usepackage{amsthm}
\usepackage{latexsym}
\usepackage{makeidx}
\usepackage{pst-all}
\usepackage{pstricks}
\usepackage{graphicx}

\newtheorem{theorem}{Theorem}[section]

\newtheorem{lemma}{Lemma}

\newcommand{\R}{\mathbb R}
\newtheorem*{teoA}{Theorem A}

\title{\Large EXISTENCE OF MULTIPLE SOLUTIONS FOR A QUASILINEAR ELLIPTIC PROBLEM  }

\author{Jorge Cossio$^a$, Sigifredo Herrón$^a$ and Carlos V\'elez$^a$}
\date{}
\begin{document}
\parindent= 0pt
\maketitle \footnotetext{$^a$ Escuela de Matem\'aticas, Universidad Nacional de Colombia Sede
Medell\'in,
 Apartado A\'ereo 3840,
 Medell\'{\i}n, Colombia. \\
 e-mail addresses: jcossio@unal.edu.co, sherron@unal.edu.co,
cauvelez@unal.edu.co\\
 This research was partially supported by
Colciencias, Fondo Nacional de Financiamiento para la Ciencia, la
Tecnolog\'\i a y la Innovaci\'on Francisco Jos\'e de Caldas. Project
``Ecuaciones diferenciales dispersivas y el\'\i pticas no lineales",
Code 111865842951.}

\begin{abstract}
In this paper we prove the existence of multiple solutions for a
quasilinear elliptic boundary value problem, when the $p$-derivative
at zero and the $p$-derivative at infinity of the nonlinearity are
greater than the first eigenvalue of the $p$-Laplace operator. Our
proof uses bifurcation from infinity and bifurcation from zero to
prove the existence of unbounded branches of positive solutions
(resp. of negative solutions). We show the existence of multiple
solutions and we provide qualitative properties of these solutions.

\textbf{Key Words and phrases}: quasilinear elliptic equations,
bifurcation theory, multiplicity of solutions.

\textbf{2010 Mathematics Subject classification}: 35B32, 35J62,
35J92.

\end{abstract}

\section{Introduction}
In this paper we study the existence of multiple solutions for the
quasilinear elliptic boundary value problem
\begin{equation}\label{PL}
\begin{cases}
\begin{aligned}
 \Delta _pu + f(u) &= 0 \quad \ \text{ in }\, \Omega,\\
 u &=0  \quad \ \text{ on }\, \partial \Omega,
\end{aligned}
\end{cases}
\end{equation}
where $\Omega \subset \mathbb{R}^N $, $N \ge 2$, is a bounded and
smooth domain, $1< p <2$,  $\Delta _p u=$ div$(|\nabla u|^{p-2}\nabla u)$ is the $p$-Laplace operator, and $f :\mathbb{R} \to \mathbb{R}$ is a nonlinear function such that $f(0)=0$ and
\begin{itemize}
\item[$(f_1)$] $ |f(t) - f(s)| \le C_f |t-s|^{p-1},\ \ \forall s,t\in \mathbb{R}$,
\item[$(f_2)$] $ f^{'}_{p}(0):= \lim _{t\rightarrow 0} \frac{f(t)}{|t|^{p-2} t} > \lambda_1 (p), $
\item[$(f_3)$] $ f^{'}_{p}(\infty):= \lim _{t\rightarrow \infty} \frac{f(t)}{|t|^{p-2} t} > \lambda_1 (p),$
\item[$(f_4)$] there exists a positive number $\alpha $
 such that $f(\alpha ) \leq 0 \leq f(-\alpha )$,
\end{itemize}
where $C_f :=\sup_{s\neq t} |f(s)-f(t)| /|s-t|^{p-1} \in
\mathbb{R}$, and $\lambda_1 (p)$ denotes the first eigenvalue of the
problem
\begin{equation}\label{EPL}
\begin{cases}
\begin{aligned}
 - \Delta _pu  &= \lambda |u|^{p-2}\ u \quad \ \text{ in }\, \Omega,\\
 u &=0  \qquad \qquad \ \ \  \text{ on }\, \partial \Omega.
\end{aligned}
\end{cases}
\end{equation}
\vskip 5pt
We call $ f^{'}_{p}(0)$ the $p$-derivative at zero and $ f^{'}_{p}(\infty)$ the $p$-derivative at infinity.
\vskip 7pt

We prove that problem (\ref{PL}) has at least four nontrivial
solutions, two of them are positive and the other two are negative.
We also found some upper and lower bounds for the $L^\infty-$ norm
of these solutions.

\begin{teoA}
\label{teoA} If $f$ satisfies $(f_1 )$, $(f_2 )$, $(f_3 )$, and $(f_4)$
then problem \eqref{PL} has at least four nontrivial solutions $u_1 ,
u_2  , v_1 ,$ and $v_2$. Moreover, solutions $u_1$ and
$u_2$ are positive on $\Omega$, and solutions $v_1$ and $v_2$ are
negative on $\Omega$. In addition,
$$\| u_2 \|_{L^{\infty }}   <\alpha < \| u_1 \|_{L^{\infty }} $$
and
$$  \| v_2
\|_{L^{\infty }} < \alpha  < \| v_1 \|_{L^{\infty }} .$$
\end{teoA}

{\bf Remark:} Actually, the argument we present below allows to
prove a  more general result: if $f$ satisfies $(f_1 )$, $(f_2 )$,
$(f_3 )$, and
\begin{itemize}
\item[$(f_4')$] there exist numbers $\alpha>0$ and $\widetilde{\alpha} <0$
 such that $f(\alpha ) \leq 0 \leq f(\widetilde{\alpha} )$,
\end{itemize}
then problem \eqref{PL} has at least four nontrivial solutions $u_1
, u_2  , v_1 ,$ and $v_2$. Moreover, solutions $u_1$ and $u_2$ are
positive on $\Omega$, and solutions $v_1$ and $v_2$ are negative on
$\Omega$. In addition,
$$\| u_2 \|_{L^{\infty }}   <\alpha < \| u_1 \|_{L^{\infty }} $$ and
$$  \| v_2
\|_{L^{\infty }} <| \widetilde{\alpha } | < \| v_1 \|_{L^{\infty }}
.$$ For the sake of simplicity, from now on we assume hypothesis
$(f_4)$ instead of $(f_4')$ (i.e. $\widetilde{\alpha }=-\alpha $).

\vskip 8pt

\vskip 8pt

 Our proof of Theorem A uses bifurcation from infinity and
bifurcation from zero, applied to the problem
\begin{equation}\label{PL2}
\begin{cases}
\begin{aligned}
 \Delta_p \, u +  \lambda f(u)&=0 \  \quad \text{ in }\, \Omega , \\
 u &=0  \quad \ \text{ on }\, \partial \Omega ,
\end{aligned}
\end{cases}
\end{equation}
where $\lambda > 0$.
\vskip 8pt

Theorem A is an extension to quasilinear equations of a result due
to J. Cossio, S. Herr\'on, and C. V\'elez (see \cite{CHV1}) for the
semilinear case. A key ingredient to extend the semilinear result to
our situation is to prove that for problem (\ref{PL2}) there exist
unbounded branches of positive solutions (resp. of negative
solutions)  emanating from the bifurcation points $(\infty ,\lambda
_1/f_p'(\infty))$ and $(0 ,\lambda _1/f_p'(\infty))$ (see Theorem
\ref{AM-lemma-around-infty} and Theorem \ref{AM-lemma-around-zero}
in Section 3 below). Theorem \ref{AM-lemma-around-infty} is very
much inspired by a corresponding result in the semilinear case due
to Ambrosetti and Hess (see \cite{AH} and Section 4.4 in \cite{AM}),
and by Theorem 4.1 in \cite{AGP}. Although our proof of Theorem
\ref{AM-lemma-around-infty} follows the ideas from \cite{AH},
\cite{AM} and \cite{AGP}, our arguments have several differences
with respect to these references, as will be better explained
 in Section 3.
Theorem \ref{AM-lemma-around-zero}, on the other
 hand, essentially comes from the ideas by Del Pino and Manásevich in
 \cite{DM}.


\vskip 8pt

The existence of solutions to quasilinear elliptic problems like
\eqref{PL2} has been widely investigated. Let us mention, besides
\cite{AGP} and \cite{DM}, papers \cite{DGTU}, \cite{Drab} and
\cite{DQ}, the books \cite{FNSS} and \cite{DKT}, and the references
therein.
A. Ambrosetti et al. in \cite{AGP} showed the existence of an unbounded branch of positive solutions of problem \eqref{PL2} emanating
from either zero or infinity  when
$f(u) \backsimeq u^{p-1}$ near $0$ or near infinity; they used a priori estimates and topological arguments. In \cite{DM} M. Del Pino and R.
Man\'asevich proved that problem \eqref{PL} has at least one nontrivial solution when
\begin{equation}
 f^{'}_{p}(0) < \lambda_1 (p) <  f^{'}_{p}(\infty).
\end{equation}
P. Dr\'abek in \cite{Drab} and S. Fucik et al. in \cite{FNSS}, focus
on the existence of solutions to problem \eqref{PL2} in the case
when $f'_p(\infty)$ is not equal to an eigenvalue of $-\Delta _p$.
By using topological arguments based on degree theory, they found
conditions that allow to show that problem \eqref{PL2} has at least
one solution for $\lambda$ either below $\lambda_1(p)$ or between
$\lambda_1(p)$ and $\lambda_2(p)$. In \cite{DGTU}, Drabek et al.
study a non-homogeneous version of problem \eqref{EPL} when
parameter $\lambda $ is near $\lambda _1 (p)$. More recently, Del
Pezzo and Quaas in \cite{DQ} generalize the results from \cite{DM}
to nonlocal problems involving fractional p-Laplacian operators.
Contrary to conditions in \cite{DM}, \cite{DGTU}, \cite{Drab},
\cite{FNSS}, and \cite{DQ}, here the $p$-derivative at zero and the
$p$-derivative at infinity are both arbitrarily greater than the
first eigenvalue of the $p$-Laplace operator.

\vskip 8pt

Regarding quasilinear equations in the radially symmetric case,
there has been a lot of research. We mention some works and refer
the reader to references therein. For instance, J. Cossio and S.
Herr\'on in \cite{CH} studied problem \eqref{PL} when $\Omega$ is
the unit ball in $\R^N$ and the $p$-derivative of the nonlinearity
at zero is greater than $\mu_{j}(p)$, the $j$-radial eigenvalue of
the $p$-Laplace operator, and the $p$-derivative at infinity is
equal to the $p$-derivative at zero. They showed that problem
\eqref{PL} has $4j-1$ radially symmetric solutions.  In such a
reference, the authors used bifurcation theory and the fact that in
the radially symmetric case \eqref{PL} reduces to an ordinary
differential equation. J. Cossio, S. Herr\'on, and C. V\'elez  in
\cite{CHV2} studied problem \eqref{PL} in the radially symmetric
case, when $\Omega$ is the unit ball in $\R^N$ and the problem is
$p$-superlinear at the origin. They proved that problem \eqref{PL}
has infinitely many solutions. The main tool that they used is the
shooting method.  M. Del Pino and R. Man\'asevich in \cite{DM}
studied the existence of multiple nontrivial solutions for a
quasilinear boundary value problem under radial symmetry; they
extended the global bifurcation theorem of P. Rabinowitz (see
\cite{R2}) and proved the existence of nontrivial solutions for that
kind of problems. In \cite{GS}, García-Melián and Sabina de Lis
study uniqueness for quasilinear problems in radially symmetric
domains.

\vskip 8pt

The paper is organized as follows: In Section 2 we establish some
lemmas which will be used to prove Theorem A. We apply a nonlinear
version of the strong maximum principle due to J. L. V\'azquez (see
\cite{V}) to prove that if $u$ is a weak solution to problem
(\ref{PL2}) then $\|u\|_{L^{\infty}} \ne \alpha$. We also apply an
interpolation theorem due to A. Le (see Theorem \ref{interp}) to
show that the function $ (u,\lambda ) \mapsto \|u \| _{L^{\infty}}$
is continuous, where $ (u,\lambda )$ is a solution of (\ref{PL2}).
In Section 3 we prove Theorem A.

\section{Preliminary results }

Let us recall the definition of weak solutions to problem \eqref{PL2}. Given $\lambda > 0$, we say a function $u\in W_0^{1, p}$
solves \eqref{PL2} in the weak sense provided that
\begin{equation}
\label{weak}
\int_{\Omega} |\nabla u|^{p-2} \nabla u \cdot \nabla v \ dx = \int_{\Omega} \lambda f(u) v\ dx, \qquad \forall v\in W_0^{1, p}.
\end{equation}

Suppose $u\in W_0^{1, p}$ solves \eqref{PL2} in the weak sense. Hypothesis $(f_1)$ implies that
\begin{equation}
\label{11}
 \Delta_p \, u = -  \lambda f(u) \leq \lambda C_f |u|^{p-1} \in L^1_{\text{Loc}} (\Omega)
\end{equation}
and
\begin{equation}
\label{12}
 - u\, \Delta_p \, u =  \lambda \, u\, f(u) \leq \lambda C_f |u|^{p}.
\end{equation}
From \eqref{11} and \eqref{12} it follows that $u\in L^{\infty}
(\Omega )$ (see, for instance, Theorem 6.2.6, p. 737, of \cite{GP}).

The following lemma states a regularity result of any weak solution $u$ of \eqref{PL2}. This result is a particular
case of a theorem of Lieberman \cite{Li} (cf. also Di Benedetto \cite{Di}).
\begin{lemma}
\label{L1} If $u\in W^{1,p}(\Omega) \cap L^{\infty}$ and $\Delta _pu \in L^{\infty}$ then
$u\in C^{1,\beta}(\overline{\Omega})$ with $\beta \in (0,1)$ and
$$ \|u\|_{C^{1,\beta}(\overline{\Omega})}\le C, $$
with $C>0$; both $\beta$ and $C$ depend only on $N, p, \lambda, \|u\|_{L^{\infty}},$ and $\| \Delta_p u\|_{L^{\infty}}.$
\end{lemma}

The next lemma is a consequence of a nonlinear version of the strong maximum principle due to J. L. V\'azquez (\cite{V}).

\begin{lemma}
\label{est-sup-norm2} If $u\in C^{1,\beta}
(\overline{\Omega } )$ is a solution of \eqref{PL2} with $\lambda
>0$, then  $\| u\| _{_{L^{\infty}}}\neq \alpha  $.
\end{lemma}
\begin{proof}
We argue by contradiction. Assume $u\in C^{1,\beta}
(\overline{\Omega } )$ is a solution of \eqref{PL2} with $\lambda
>0$ such that  $\| u\| _{_{L^{\infty}}}= \alpha  $. Since
$$\lambda |f(u)| \le \lambda C_f \|u\|_{L^{\infty}}^{p-1},$$
it follows that
$$ -  \Delta _pu = \lambda f(u)\in L^2_{loc}(\Omega).$$

We consider the function $\alpha - u \in C^{1}(\overline{\Omega})$, $\alpha - u \ge 0$ in $\Omega$.
$$ \Delta _p(\alpha -u) = \text{div} (|\nabla (\alpha - u)|^{p-2}\nabla (\alpha -u))= - \text{div} (|\nabla u|^{p-2}\nabla u)\in L^2_{loc}(\Omega).$$

Since $f(\alpha) \le 0$, we see that
\begin{equation}\label{Des1}
\begin{aligned}
 \Delta _p(\alpha -u) = \lambda f(u)= \lambda f(\alpha -(\alpha -u)) &\le  \lambda f(\alpha -(\alpha -u)) - \lambda f(\alpha) \\
  &\le  \lambda |f(\alpha -(\alpha -u)) - f(\alpha)| \\
  &\le \lambda C_f |\alpha -u|^{p-1}.
\end{aligned}
\end{equation}

Let us define $\xi: \R_+ \to \R$ by $\xi(s) = \lambda C_f s^{p-1}.$
We see that $\xi$ is continuous, increasing function, such that
$\xi(0)=0$ and
$$ \int_0^1 \frac{1}{(s\ \xi(s))^{\frac{1}{p}}} ds = c  \int_0^1 \frac{1}{(s\ s^{p-1})^{\frac{1}{p}}} ds = c \ln s]_0^1 = +\infty.$$

Hence, V\'azquez maximum principle (see \cite{V}) implies $\alpha -u
>0$ in $\Omega$, i.e. $u<\alpha $ in $\Omega$. Thus $\| u\|
_{_{L^{\infty}}}< \alpha  $, which contradicts our initial
assumption.
\end{proof}

In the proof of Theorem A inequalities \eqref{L2-1} and \eqref{L2-2}
of the following lemma will play an important role. These
inequalities essentially come from the arguments leading to
regularity results due to \cite{Di}, \cite{T} and \cite{Li}.

\begin{lemma}
\label{L2} There exist positive constants $K_1:= K_1(|\Omega|, N,
C_f, p, \lambda )$ and

$K_2:= K_2 (|\Omega|, N, C_f, p, \lambda )$ such that
 if $u\in W_0^{1,p}(\Omega)$ is a solution of \eqref{PL2} then
\begin{equation}
\label{L2-1} \| u\| _{W_0^{1,p}} \leq  K_1\   \| u\|_{L^{\infty}}
\end{equation}
and
\begin{equation}
\label{L2-2} \| u\| _{L^{\infty}} \leq K_2\  \| u\| _{W_0^{1,p}}.
\end{equation}
Moreover, $K_1$ and $K_2$ are bounded if $\lambda$ is bounded.
\end{lemma}
\begin{proof}
Let $u\in W_0^{1,p}(\Omega)$ be a solution of \eqref{PL2}. Using the definition of weak solution and hipothesis $(f_1)$ it follows that
\begin{equation}
\label{LL2} \|u\|^p_{W_0^{1,p}} = \int_{\Omega} |\nabla u|^{p-2}
\nabla u \cdot \nabla u \ dx = \int_{\Omega} \lambda u\, f(u) \ dx
\leq |\Omega| \lambda C_f \|u\|^p_{L^{\infty}}.
\end{equation}
Defining $K_1:= (|\Omega| \lambda C_f )^{\frac{1}{p}}$, inequality
\eqref{L2-1} follows from \eqref{LL2}.

\vskip 8pt

Using \eqref{11}, \eqref{12}, and a \emph{boot-strap} argument (see,
for instance, the proof of Theorem 6.2.6 in \cite{GP}) we get that
there exists a constant $K:=K(|\Omega|, N, C_f , p, \lambda )>0$,
which is bounded when $\lambda$ is bounded,  such that
\begin{equation}
\label{13}
\|u\|_{L^{\infty}} \leq K\, \|u\|_{L^{p_0}},
\end{equation}
where $p_0= \frac{Np}{N-p}$ is a critical Sobolev exponent. Since
$W_0^{1,p} (\Omega )$ is continuously embedded in $ L^{p_0} (\Omega
)$, we see that
\begin{equation}
\label{14} \|u\|_{L^{p_0}} \leq c_0 \|u\|_{W_0^{1,p}},
\end{equation}
for a constant $c_0 >0$. From \eqref{13} and \eqref{14} we get a
constant $K_2>0$ satisfying inequality \eqref{L2-2}. The proof of
Lemma \ref{L2} is complete.
\end{proof}

Let us define
\begin{equation}
S=\{ (u,\lambda )\in W_0^{1, p} (\Omega )\times \mathbb{R}:
\, \, u\neq 0 \, \,  \text{and} \, \,  u= (-\Delta_p)^{-1} (\lambda \
f(u)) \},
\end{equation}
where the inverse of the $p$-Laplace operator $L:= (-\Delta_p)^{-1} :
L^{\infty} (\Omega ) \longrightarrow  C^{1,\beta}(\overline{\Omega }
)$ is known to be a continuous and compact mapping (see \cite{Di}
and \cite{Li}). We will make use of the next lemma in the proof of
Theorem A.
\begin{lemma}
\label{continuity-sup-norm} The function $\mathcal{N}_\infty :
\overline{S}  \subset W_0^{1, p} (\Omega )\times \mathbb{R} \longrightarrow
\mathbb{R}$ defined as $ (u,\lambda ) \mapsto \|u \| _{L^{\infty}}$
is continuous.
\end{lemma}
\begin{proof}
We commence by observing that if $(u, \lambda )\in W_0^{1, p} (\Omega )\times \mathbb{R}$ is a limit point of $S$ then $u= (-\Delta_p)^{-1} (\lambda f(u))$, and so $\|u \| _{L^{\infty}}$ is
well-defined on all $\overline{S}$. Let us take $(u, \lambda _u),
(v_n, \lambda _{v_n}) \in \overline{S}$ such that $(v_n, \lambda _{v_n}) \to (u, \lambda _u)$. Let us try to estimate
$|\mathcal{N}_\infty (v_n, \lambda
_{v_n}) - \mathcal{N}_\infty (u, \lambda _u)|$.

\begin{equation}
\label{L3-1}
\begin{aligned}
\|{v_n}-u\|_{\infty} &= \| L(\lambda_{v_n} f({v_n})) - L(\lambda_u f(u))\|_{\infty} \\
    &=  \|\lambda_{v_n}^{\frac{1}{p-1}} L(f({v_n})) - \lambda_u^{\frac{1}{p-1}} L(f(u))\|_{\infty} \\
    &\le \lambda_{v_n}^{\frac{1}{p-1}} \| L(f({v_n})) - L(f(u))\|_{\infty} + |\lambda_{v_n}^{\frac{1}{p-1}}- \lambda_u^{\frac{1}{p-1}}|\ \| L(f(u))\|_{\infty}.
\end{aligned}
\end{equation}
Let us define
\begin{equation}
\label{L3-2}
u^* = L(f(u)) \ \text{and} \  {v_n}^* = L(f({v_n})).
\end{equation}
To estimate $\| L(f({v_n})) - L(f(u))\|_{\infty}$ we will need an
interpolation type inequality between $C^{1,0}(\overline{\Omega } ),
C^{1,\beta}(\overline{\Omega}),$ and $W^{1, p} (\Omega)$. We will
use the following interpolation theorem due to A. Le (\cite{Le}).

\begin{theorem}
\label{interp}  There exist constants $c>0$ and $0< \theta < 1$ such
that for any $u\in C^{1,\beta}(\overline{\Omega}) \cap W^{1, p}
(\Omega)$,
\begin{equation}
\|u\|_{1,0} \le c\ \|u\|_{C^{1,\beta}}^{1- \theta} \|u\|_{W^{1, p}}^{\theta}.
\end{equation}
\end{theorem}
Since $u{^*}, {v_n}{^*}\in C^{1,\beta}(\overline{\Omega}) \cap
W_0^{1, p} (\Omega)$, by using the previous theorem and Poincarè
inequality we see that

\begin{equation}
\label{L3-3}
\begin{aligned}
\| L(f({v_n})) - L(f(u))\|_{\infty} &\le  \| L(f({v_n})) - L(f(u))\|_{1,0} \\
&= \|{v_n}{^*} - u{^*}\|_{1,0} \\
    &\le  c\|{v_n}{^*} - u{^*}\|_{C^{1,\beta}}^{1-\theta} \|{v_n}{^*} - u{^*}\|_{W_0^{1, p}}^{\theta}.
\end{aligned}
\end{equation}

We claim that there exists $C>0$ such that
\begin{equation}
\label{L3-4}
\|{v_n}{^*} - u{^*}\|_{C^{1,\beta}}^{1-\theta} \le C.
\end{equation}
To prove (\ref{L3-4}) we first show that there exists $M_1>0$ such that $u, v_n\in B_{M_1}^{\infty}$, the ball with radius $M_1$ in $L^{\infty}$. Since ${v_n}\to u$ in $W_0^{1, p}$, $\|{v_n}\|_{W_0^{1, p}}$, $\|{v_n}\|_{L^{p_0}}$, $\|u\|_{W_0^{1, p}}$, and $\|u\|_{L^{p_0}}$ are bounded by a constant. From Lemma \ref{L2} we have
\begin{equation}
\label{L3-5}
\|u\|_{L^{\infty}} \le  K \  \|u\|_{W_0^{1, p}} \quad \text{and} \quad \|{v_n}\|_{L^{\infty}} \le K \   \|{v_n}\|_{W_0^{1, p}},
\end{equation}
where $K$ denotes a positive constant. Thus, there exists $M_1>0$ such that
\begin{equation}
\label{L3-6}
u, {v_n}\in B_{M_1}^{\infty}.
\end{equation}
Combining (\ref{L3-6}),
\begin{equation}
\label{L3-7} \|f(u)\|_{L^{\infty}} \le C_f
\|u\|_{L^{\infty}}^{p-1},\quad \text{and}\quad
\|f({v_n})\|_{L^{\infty}} \le C_f \|{v_n}\|_{L^{\infty}}^{p-1},
\end{equation}
we see that there exists $M_2>0$ such that
\begin{equation}
\label{L3-8}
\|f(u)\|_{L^{\infty}} \le M_2\quad \text{and}\quad
\|f({v_n})\|_{L^{\infty}} \le M_2.
\end{equation}
As we mentioned above, from the regularity results the inverse of the p-Laplace operator
\begin{equation}
\label{L3-9}
L:= (-\Delta_p)^{-1} : L^{\infty} (\Omega )
\longrightarrow  C^{1,\beta}(\overline{\Omega } )
\end{equation}
is a continuous and compact mapping. An immediate consequence of (\ref{L3-8}) and (\ref{L3-9}) is that there exists $M>0$ such that
\begin{equation}
\label{L3-10}
\|u^*\|_{C^{1,\beta}} \le M\quad \text{and}\quad
\|{v_n}{^*}\|_{C^{1,\beta}} \le M.
\end{equation}
Now (\ref{L3-10}) implies that there exists $C>0$ such that
\begin{equation}
\label{L3-11}
\|{v_n}{^*} - u{^*}\|_{C^{1,\beta}}^{1-\theta} \le C,
\end{equation}
which proves (\ref{L3-4}). From (\ref{L3-3}), (\ref{L3-4}), and
(\ref{L3-11}) we see that
\begin{equation}
\label{L3-12}
\| L(f({v_n})) - L(f(u))\|_{\infty} \le  C\  \|{v_n}{^*} -
u{^*}\|_{W_0^{1, p}}^{\theta}.
\end{equation}

Since
\begin{equation}
{v_n}{^*} =\frac{v_n}{{\lambda_{v_n}}^{\frac{1}{p-1}}} \quad \text{and}  \quad u^* =\frac{u}{\lambda_u^{\frac{1}{p-1}}}
\end{equation}
it follows that
\begin{equation}
\label{L3-13}
\begin{aligned}
\| L(f({v_n})) - L(f(u))\|_{\infty} &\le  C\  \| {v_n} {{\lambda_{{v_n}}}^{-\frac{1}{p-1}}} - {u} {\lambda_u^{-\frac{1}{p-1}}}\|_{W_0^{1, p}}^{\theta} \\
& = \frac{C}{{\lambda_u^{\frac{1}{p-1}}}{{\lambda_{v_n}^{\frac{1}{p-1}}}}}
 \| v_n \lambda_u^{\frac{1}{p-1}} - u {\lambda_{{v_n}}^{\frac{1}{p-1}}} \|_{W_0^{1, p}}^{\theta} \\
& = \frac{C}{{\lambda_u^{\frac{1}{p-1}}}{{\lambda_{{v_n}}^{\frac{1}{p-1}}}}}
 \| v_n (\lambda_u^{\frac{1}{p-1}} - {\lambda_{{v_n}}^{\frac{1}{p-1}}}) + {\lambda_{{v_n}}^{\frac{1}{p-1}}} (v_n - u) \|_{W_0^{1, p}}^{\theta} \\
& \le \frac{C}{{\lambda_u^{\frac{1}{p-1}}}{{\lambda_{v_n}}^{\frac{1}{p-1}}}}
 \| v_n\|_{W_0^{1, p}} |\lambda_u^{\frac{1}{p-1}} - {{\lambda_{{v_n}}^{\frac{1}{p-1}}}| + |{\lambda_{{v_n}}}|^{\frac{1}{p-1}}}
 \|v_n - u \|_{W_0^{1, p}}^{\theta}. \\
\end{aligned}
\end{equation}
Because the sequences $\{\|{v_n}\|_{W_0^{1, p}}\}$ and $\{\lambda_{{v_n}}\}$ are bounded, there exists $C_1$ such that
\begin{equation}
\label{L3-14} \| L(f({v_n})) - L(f(u))\|_{\infty} \le  C_1
|\lambda_u^{\frac{1}{p-1}} - {\lambda_{{v_n}}^{\frac{1}{p-1}}}| +
\|v_n - u \|_{W_0^{1, p}}^{\theta}.
\end{equation}

From (\ref{L3-1}), (\ref{L3-14}), ${\lambda_{{v_n}}}\to \lambda_u$, and ${v_n}\to u$ in $W_0^{1, p}$ it follows that
\begin{equation}
|\mathcal{N}_\infty (v_n, \lambda
_{v_n}) - \mathcal{N}_\infty (u, \lambda _u)| \longrightarrow 0,
\end{equation}

which proves the lemma.
\end{proof}

\section{Proof Theorem A}

Let $f$ be a function satisfying the hypotheses ($f_1$) - ($f_4$).
Because of regularity theory (see \cite{Di} and \cite{Li}), the
problem of finding  solutions $u\in C^{1,\beta}(\overline{\Omega })$
to \eqref{PL2} is equivalent to find elements $u\in W_0^{1,p}
(\Omega )$ such that
\begin{equation}
\label{laplacian-inverse}
u=  (-\Delta_p)^{-1} (\lambda f(u)).
\end{equation}
We will prove that there are nontrivial solutions of
\eqref{laplacian-inverse} when $\lambda =1$, i.e. four nontrivial
solutions of \eqref{PL}.

\vskip 8pt

Let $f^{+}:\mathbb{R} \to \mathbb{R}$ be defined as $f^{+} (t)=f(t)$
for $t\geq 0$, and $f^{+} (t)=0$ for $t<0$. Similarly, let
$f^{-}:\mathbb{R} \to \mathbb{R}$ be defined as $f^{-} (t)=f(t)$ for
$t\leq 0$, and $f^{-} (t)=0$ for $t>0$. We observe that $f$ can be
written as $$f(t)=f_p'(\infty )|t|^{p-2}t+g(t),$$ where
$g(t)/|t|^{p-2}t \longrightarrow 0$ as $|t|\rightarrow \infty $, and
also
$$f(t)=f_p'(0)|t|^{p-2}t+\widehat{g}(t),$$
where $\widehat{g}(t)/|t|^{p-2}t \longrightarrow 0$ as $t\rightarrow
0$. From Vázquez maximum principle (see \cite{V}), we have the
following lemma.

\begin{lemma}
\label{u-sln-f-mas-tau} If $u\in W_0^{1,p} (\Omega )\setminus \{0
\}$ satisfies $u=(-\Delta_p)^{-1} (\lambda f^{+}(u) +\tau )$, where
$\lambda >0$ and $\tau \geq 0$, then $u\in
C^{1,\beta}(\overline{\Omega})$ for some $\beta \in (0,1)$, $u>0$ on
$\Omega $ and $\frac{\partial u }{\partial \overrightarrow{n}} <0$
 (where $\overrightarrow{n}$ denotes the outer unit normal on
$\partial \Omega$).
\end{lemma}

\textbf{Remark:} taking $\tau =0$ in the previous lemma, we observe
that if $u \in W_0^{1,p}(\Omega )$ is a solution of
\begin{equation}
\label{laplacian-inverse-truncation-zero-+} u= (-\Delta_p)^{-1}
(\lambda f^{+}(u))
\end{equation}
and $\lambda >0$, then $u>0$ on $\Omega$. Thus $u$ satisfies
\eqref{laplacian-inverse}, i.e. $(u,\lambda )\in S$. In a similar
way, if $u \in W_0^{1,p}(\Omega ) $ is a nontrivial solution of
\begin{equation}
\label{laplacian-inverse-truncation-zero--} u= (-\Delta_p)^{-1}
(\lambda f^{-}(u))
\end{equation}
and $\lambda >0$, then $u<0$ on $\Omega$. Thus $u$ satisfies
\eqref{laplacian-inverse}, i.e. $(u,\lambda )\in S$.



We define
$$S^{+}=\{ (u,\lambda )\in W_0^{1,p} (\Omega )\times \mathbb{R}: \, \,
u\neq 0 \, \, \text{and} \, \, u= (-\Delta_p)^{-1} ( \lambda
f^{+}(u)) \}$$ and
$$S^{-}=\{ (u,\lambda )\in W_0^{1,p} (\Omega )\times \mathbb{R}:
\, \, u\neq 0 \, \, \text{and} \, \,  u= (-\Delta_p)^{-1} (\lambda
f^{-}(u)) \}.$$

As we mentioned above, we use bifurcation theory (see \cite{Ra},
\cite{R}, \cite{R2} and \cite{AM}) to prove Theorem A. Let us recall
that, in our framework, $(0,\lambda ^* )$ is a \emph{bifurcation
point from zero} for equation $u= (-\Delta_p)^{-1} ( \lambda
f^{+}(u))$ if $(0,\lambda ^* ) \in \overline{S^{+}}$ or,
equivalently, if there exists a sequence $\{ (u_n ,\lambda _n) \}
_n$ in $S^{+}$ which converges to $(0,\lambda
 ^{*} )$. Also, $(\infty,\lambda ^* )$ or simply $\lambda ^{*}$ is a \emph{bifurcation
point from infinity} for equation $u= (-\Delta_p)^{-1} ( \lambda
f^{+}(u))$ if  there exists a sequence $\{ (u_n ,\lambda _n) \} _n$
in $S^{+}$ such that $\lambda _n \longrightarrow \lambda ^{*}$ and
$\|{u_n}\|_{W_0^{1, p}} \longrightarrow \infty$ as $n\rightarrow
\infty$. Similar definitions apply for equation $u= (-\Delta_p)^{-1}
( \lambda f^{-}(u))$.

\vskip 8pt

First we present an argument using bifurcation from infinity to show
the existence of two one-sign solutions of (\ref{PL}). Secondly, we
use bifurcation from zero to show the existence of two additional
one-sign solutions. At the end of this section we include a
bifurcation diagram which summarizes the arguments presented below.

\subsection{Bifurcation from infinity}

Let us define $\Psi _{+} : W_0^{1,p}(\Omega ) \times \mathbb{R}
\longrightarrow W_0^{1,p}(\Omega )$ by
\begin{equation*}
  \begin{aligned} \Psi _{+}(z,\lambda ) =\begin{cases}
z-  \| z \|^{2} (-\Delta_p)^{-1} \left[\lambda f^{+}\left(
\frac{z}{\| z \|^{2} } \right) \right]  \quad \text{if }
& z\neq 0,\\
\hspace*{3cm}  0\phantom{aaaaaaaaaaaaa}  \text{if} &z=0 ,
\end{cases}
\end{aligned}
\end{equation*}
and $\Psi _{-}$ in the same way, changing $f^{+}$ by $f^{-}$. The
following result will be used to prove the existence of two one-sign
solutions for problem \eqref{PL}.

\begin{theorem}
\label{AM-lemma-around-infty}  $(\infty, \lambda _1 /f_p'(\infty))$ is the unique bifurcation point from infinity for equation \eqref{laplacian-inverse-truncation-zero-+}. More precisely, there exists a connected component $\Sigma
_{\infty}^{+} $ of $S^{+}$ bifurcating from $(\infty ,\lambda _1/f_p'(\infty) )$ which corresponds to an unbounded connected
component $\Gamma _{\infty }^{+}$ of
$$\Gamma ^{+} =\{ (z,\lambda )\in W_0^{1,p} (\Omega )\times \mathbb{R}: \, \,
z\neq 0 \, \, \text{and} \, \, \Psi_{+} (z,\lambda)=0 \}, $$
emanating from the trivial solution of $\Psi_{+} (z,\lambda)=0$ at
$(0,\lambda _1 /f_p'(\infty ) )$. Analogously, the point $(\infty, \lambda _1/f_p'(\infty ))$ is the unique  bifurcation from infinity
for equation \eqref{laplacian-inverse-truncation-zero--}. More precisely, there exists a connected component $\Sigma
_{\infty}^{-} $ of $S^{-}$ bifurcating from $(\infty ,\lambda _1
/f_p'(\infty) )$ which corresponds to an unbounded connected
component $\Gamma _{\infty }^{-}$ of
$$\Gamma ^{-} =\{ (z,\lambda )\in W_0^{1,p} (\Omega )\times \mathbb{R}: \, \,
z\neq 0 \, \, \text{and} \, \, \Psi_{-} (z,\lambda)=0 \}, $$
emanating from the trivial solution of $\Psi_{-} (z,\lambda)=0$ at
$(0,\lambda _1 /f_p'(\infty ) )$.
\end{theorem}
 \textbf{Remark:} As we mentioned in the introduction above, Theorem \ref{AM-lemma-around-infty} is inspired by a corresponding
 result in the semilinear case due to Ambrosetti and Hess
 (see \cite{AH} and Section 4.4 in \cite{AM}), and by Theorem 4.1 in \cite{AGP} (see also
 \cite{DM}). The proof we
 present below closely follows the ideas from \cite{AH}, \cite{AM} and \cite{AGP}, but
 our arguments have several differences with respect to these
 references. First, as expected, a lot of technicalities arise when
 trying to adapt the $\Delta -$approach from \cite{AH}
 and \cite{AM} to the $\Delta _p$ nonlinear operator. Second, our
 hypotheses on $f$ slightly differ from those in Theorem 4.1 of \cite{AGP} (ours are a little less restrictive near
 infinity) and, in the proof presented in \cite{AGP}, several details are omited. And third, our
 choice of functional spaces is different from both references. For the sake of completeness we include full details here.

 In
 order to prove Theorem \ref{AM-lemma-around-infty} we need the
 following lemmas.

\begin{lemma}
\label{4.15 AM} Let $J\subset \mathbb{R}^+$ be a compact interval
such that $\lambda _{\infty } :=\lambda _1 /f_p'(\infty) \notin J$.
Then
\begin{itemize}
\item[a)] There exists $r>0$ such that  $u\neq (-\Delta_p)^{-1}
(\lambda f^{+}(u))$ for every $\lambda \in J$ and every $u\in
W_0^{1,p}(\Omega )$ with $\| u\| _{W_0^{1,p}} \geq r$.
\item[b)] $(\infty, \lambda _1 /f_p'(\infty))$ is the unique bifurcation point from infinity for equation
\eqref{laplacian-inverse-truncation-zero-+}.
\item[c)] $i(\Psi _{+}(\cdot,\lambda ),0)=1$ for every $\lambda <\lambda _{\infty
}$ (here, $i(\Psi _{+}(\cdot,\lambda ),0)$ denotes the index of $\Psi
_{+}(\cdot,\lambda )$ with respect to zero).
\end{itemize}
\end{lemma}

\begin{proof}
In order to prove a) we argue by contradiction. Assume there exist a
sequence $\{ \lambda _n \} _n  \subset J$ and a sequence $\{ u_n \}
_n \subset W_0^{1,p} (\Omega )$ such that $\| u_n \| _{W_0^{1,p}}
\longrightarrow +\infty$ and
\begin{equation} \label{15}
u_n =(-\Delta_p)^{-1} (\lambda _n f^{+}(u_n ))\quad \text{for every}
\quad n\in \mathbb{N}.
\end{equation}
Because of Lemma \ref{u-sln-f-mas-tau}, $u_n \geq 0$ for every $n$.
Dividing \eqref{15} by $\| u_n \| _{W_0^{1,p}}$ we get
\begin{equation} \label{16}
\frac{u_n }{ \| u_n \| _{W_0^{1,p}} } =(-\Delta_p)^{-1} \left(
 \frac{ \lambda _n f_p'(\infty )u_n ^{p-1} +\lambda _n g(u_n ) }{\| u_n \|^{p-1} _{W_0^{1,p}} } \right) \quad \text{for every} \quad n\in
\mathbb{N},
\end{equation}
where $g(t)/|t|^{p-2}t \longrightarrow 0$ as $t\rightarrow \infty $.
What follows is a standard compactness argument. Indeed, since $\{
 u_n / \| u_n \| _{W_0^{1,p}} \} _n$ is a bounded sequence in
$W_0^{1,p} (\Omega )$, there exists a subsequence, for which we keep
the same notation, $\overline{v} \in W_0^{1,p} (\Omega )$ and $h\in
L^p (\Omega )$ such that
\begin{equation}\label{seq u-n lema 4.15 }
\begin{cases}
\begin{aligned}
 \frac{u_n }{ \| u_n \| _{_{W_0^{1,p}} }} \rightharpoonup \
&\overline{v} \quad    \text{weakly in}\quad W_0^{1,p} (\Omega ) \\
 \frac{u_n }{ \| u_n \| _{_{W_0^{1,p}} }} \rightarrow \ &\overline{v}
\quad    \text{strongly in}\quad L^p (\Omega ) \\
 \frac{u_n (x) }{ \| u_n \| _{_{W_0^{1,p}} }} \rightarrow \
 &\overline{v} (x)
\quad   \text{a.e. } x\in \Omega  \\
 \frac{u_n (x) }{ \| u_n \| _{_{W_0^{1,p}} }} \leq \ & h(x) \quad   \text{a.e. } x\in
 \Omega.
\end{aligned}
\end{cases}
\end{equation}
Now, let us verify $u^{p-1} _n / \| u_n \|^{p-1} \rightharpoonup
\overline{v}\ ^{p-1} $ and $g(u _n ) / \| u_n \|^{p-1}
\rightharpoonup 0 $ weakly in $L^{p' } (\Omega )$, where $1/p
+1/p^{'} =1$. Let $\omega \in L^{p } (\Omega )$. Then, from
\eqref{seq u-n lema 4.15 },
$$\frac{u^{p-1} _n (x) \omega (x) }{ \| u_n \|^{p-1}} \longrightarrow
\overline{v}\ ^{p-1} (x) \omega (x)\  \text{ a.e. } x\in \Omega
\quad \text{ and }\quad  \frac{u^{p-1} _n  \omega  }{ \| u_n
\|^{p-1}} \leq |h |^{p-1} \omega.$$

Since $h\in L^p (\Omega )$, $|h |^{p-1} \in L^{p'} (\Omega )$.
Hence, dominated convergence theorem implies that
$$\int _{\Omega } \frac{u^{p-1} _n  \omega  }{ \| u_n
\|^{p-1}} \, dx \longrightarrow \int _{\Omega } \overline{v}\ ^{p-1}
\omega \, dx \ \text{ as } n\rightarrow \infty .$$ Since this holds
true for every $\omega \in L^{p } (\Omega )$, Riesz representation
theorem guarantees that $u^{p-1} _n / \| u_n \|^{p-1}
\rightharpoonup \overline{v}\ ^{p-1} $ weakly in $L^{p' } (\Omega
)$. In order to verify $g(u _n ) / \| u_n \|^{p-1} \rightharpoonup 0
$ weakly in $L^{p' } (\Omega )$, we take $\varepsilon >0$ and then,
since $g(t)/|t|^{p-2}t \longrightarrow 0$ as $t\rightarrow \infty $,
there exists $M_{\varepsilon} >0$ such that
\begin{equation} \label{17}
t>M_{\varepsilon} \ \Longrightarrow  \ |g(t)|< \varepsilon t^{p-1}.
\end{equation}
Given $n\in \mathbb{N}$, we observe that
\begin{equation} \label{18}
\int_{\Omega } \frac{g(u_n )}{\| u_n \|^{p-1}} \omega \, dx =
\int_{|u_n | >M_{\varepsilon } } \frac{g(u_n )}{\| u_n \|^{p-1}} \omega
\, dx + \int_{|u_n | \leq M_{\varepsilon } } \frac{g(u_n )}{\| u_n
\|^{p-1}} \omega \, dx.
\end{equation}
Regarding the first integral on the right-hand side of \eqref{18},
from \eqref{17}, Hölder inequality, and the continuity of the
embedding, we get
\begin{equation}\label{19}
\begin{aligned}
\left| \int_{|u_n | >M_{\varepsilon } } \frac{g(u_n )}{\| u_n \|^{p-1}}
\omega \, dx \right| &= \int_{|u_n | >M_{\varepsilon } } \frac{|g(u_n
)|}{u^{p-1}_n } \frac{u^{p-1}_n }{\| u_n \|^{p-1}} |\omega | \, dx
\leq \varepsilon \int_{|u_n | >M_{\varepsilon } } \frac{u^{p-1}_n }{\| u_n
\|^{p-1}} |\omega | \, dx \\
&\leq \varepsilon \| \omega \| _{L^p}  \left\| \frac{u^{p-1}_n }{\| u_n
\|^{p-1}} \right\| _{L^{p'}} \leq C\varepsilon \| \omega \| _{L^p}.
\end{aligned}
\end{equation}
With respect to the second integral on the right-hand side of
\eqref{18}, we have
\begin{equation}\label{20}
\left| \int_{|u_n | \leq M_{\varepsilon } } \frac{g(u_n )}{\| u_n
\|^{p-1}} \omega \, dx \right| = \|g\|_{L^{\infty }[0,M_{\varepsilon
}]}\int_{|u_n | \leq M_{\varepsilon } } \frac{|\omega |}{\| u_n
\|^{p-1}}  \, dx \leq \frac{\|g\|_{L^{\infty }[0,M_{\varepsilon }]}
}{\| u_n \|^{p-1}} \| \omega \|_{L^1 }.
\end{equation}
Since $\varepsilon >0$ is fixed, $\| g \|_{L^{\infty }[0,M_{\varepsilon
}]}$ is fixed. The right-hand side of \eqref{20} tends to zero as
$n\rightarrow \infty$, because $\| u_n \| _{W_0^{1,p}}
\longrightarrow +\infty$. Thus, from \eqref{18}, \eqref{19},
\eqref{20}, and the fact that $\omega \in L^{p} (\Omega )$ is
arbitrary, we conclude $g(u _n ) / \| u_n \|^{p-1} \rightharpoonup 0
$ weakly in $L^{p' } (\Omega )$.

\vskip 8pt

We then have that the argument on the right-hand side in \eqref{16}
converges weakly to $\overline{\lambda} f_p'(\infty ) \overline{v}$
in $L^{p'} (\Omega )$, for some $\overline{\lambda} \in J $. As
$(-\Delta_p)^{-1} : L^{p'} (\Omega ) \longrightarrow W^{1,p} _0
(\overline{\Omega } )$ is compact, from \eqref{16} we get a further
subsequence $\{ \frac{ u_n }{ \| u_n \| _{W_0^{1,p}}} \} _n$ such
that
\begin{equation}\label{21}
\frac{u_n }{ \| u_n \| _{W_0^{1,p}} } =(-\Delta_p)^{-1} \left(
 \frac{ \lambda _n f_p'(\infty )u_n ^{p-1} +\lambda _n g(u_n ) }{\| u_n \|^{p-1} _{W_0^{1,p}} }
 \right) \longrightarrow (-\Delta_p)^{-1} \left(  \overline{\lambda} f_p'(\infty ) \overline{v} \right) \  \ \text{as }\ n\rightarrow \infty
\end{equation}
strongly $W^{1,p} _0 (\overline{\Omega } )$. From \eqref{seq u-n
lema 4.15 } and \eqref{21} we conclude
\begin{equation}\label{21-b}
(-\Delta_p)^{-1} \left(  \overline{\lambda} f_p'(\infty )
\overline{v} \right) =\overline{v}.
\end{equation}
We claim $\overline{v} \neq 0$. Let us denote $v_n :=u_n / \| u_n \|
_{W_0^{1,p}} $ for each $n$. From \eqref{16}, we have that
\begin{equation}\label{22}
 -\Delta_p v_n=
  \lambda _n f_p'(\infty )v_n ^{p-1} +\frac{\lambda _n g(u_n ) }{\| u_n \|^{p-1} _{W_0^{1,p}} }  \quad \text{for every} \quad n\in
\mathbb{N},
\end{equation}
in the weak sense. Multiplying \eqref{22} by $v_n$ and integrating
we get
\begin{equation}\label{23}
1=\| v_n \|^{p} _{W_0^{1,p}}=
  \lambda _n f_p'(\infty ) \| v_n \|^{p} _{L^p} +\int _{\Omega } \frac{\lambda _n g(u_n ) }{\| u_n \|^{p-1} _{W_0^{1,p}} } v_n \, dx  \quad \text{for every} \quad n\in
\mathbb{N}.
\end{equation}
By virtue of \eqref{seq u-n lema 4.15 } and the compactness of
$J\subset \mathbb{R}^+$, the first term on the right-hand-side in
\eqref{23} tends to $\overline{\lambda} f_p'(\infty ) \|
\overline{v} \|^{p} _{L^p}$. Arguing as above, when we proved $g(u
_n ) / \| u_n \|^{p-1} \rightharpoonup 0 $ weakly in $L^{p' }
(\Omega )$, one can show
$$\int _{\Omega } \frac{\lambda _n g(u_n ) }{\| u_n \|^{p-1} _{W_0^{1,p}} } v_n \, dx \longrightarrow 0 \ \text{as}\ n\rightarrow \infty $$
(double checking \eqref{19} and \eqref{20}, replacing $\omega$ by
$v_n$, one can observe that the same argument can be carried out
provided that sequence $\{ \| v_n \|^{p} _{L^p} \} _n$ be bounded,
which is our case). Thus, taking limit in \eqref{23}, we conclude
$1=\overline{\lambda} f_p'(\infty ) \| \overline{v} \|^{p} _{L^p}$
and so $\overline{v} \neq 0$.

\vskip 6pt

Therefore \eqref{21-b} means $\overline{\lambda} f_p'(\infty )$ is
an eigenvalue of $-\Delta_p $ and $\overline{v}$ is an associated
eigenfunction. This is absurd since $\overline{v}\geq 0$ (from
\eqref{seq u-n lema 4.15 }) and $\overline{\lambda} f_p'(\infty
)\neq \lambda _1$ (since $\overline{\lambda} \in J$ and $\lambda _1
/f_p'(\infty) \notin J$). This contradiction completes our proof of
a).

\vskip 8pt

To prove b), again we argue by contradiction. Assume there is a
bifurcation point $\overline{\lambda}$ from $\infty$ such that
$\overline{\lambda} \neq \lambda _{\infty }$. Let $J \subset
\mathbb{R}^+$ be a compact interval such that $\overline{\lambda}
\in J$ and $\lambda _{\infty } \notin J$. Then, there exists a
sequence $\{ (u_n ,\lambda _n )\} \subset S^{+}$ such that $\| u_n
\| _{W_0^{1,p}} \longrightarrow +\infty$ and  $\lambda _n \in J$ for
large $n\in \mathbb{N}$. But this contradicts a).

\vskip 6pt

We now prove c). Let $\lambda <\lambda _{\infty }$. Consider
$J=[0,\lambda ]$. For every $t\in [0,1]$ we have $t\lambda \in J$.
From a) it follows that
$$u- (-\Delta_p)^{-1} (t\lambda f^{+}(u))\neq 0$$
for every $\lambda \in J$ and every $u\in W_0^{1,p}(\Omega )$ with
$\| u\| _{W_0^{1,p}} \geq r$. For such an $u$, taking $z= u/\| u\|^2
_{_{W_0^{1,p}}}$, we get
$$z- \| z\|^2 (-\Delta_p)^{-1} (t\lambda f^{+}\left( z/\| z\|^2 \right) )\neq 0 $$
for every $z\in W_0^{1,p}(\Omega )$ such that $\| z\| _{W_0^{1,p}}
\leq 1/r$. Hence, $\Psi _{+}( z,t\lambda ) \neq 0$ for every $z\in
W_0^{1,p}(\Omega )$ such that $0< \| z\| _{W_0^{1,p}} \leq 1/r$. Let
us define the homotopy $H: [0,1]\times W_0^{1,p}(\Omega )
\longrightarrow W_0^{1,p}(\Omega )$ by $H(t,u)= \Psi _{+}(
u,t\lambda )$. Using Leray-Schauder degree invariance under
homotopies, we get
$$\deg(H(1,\cdot), B_{1/r} (0),0)=deg(H(0,\cdot), B_{1/r} (0),0)$$
equivalently
$$\deg(\Psi _{+}(
\cdot,\lambda ), B_{1/r} (0),0)=deg(I, B_{1/r} (0),0)=1.$$

\end{proof}

\begin{lemma}
\label{4.16 AM} The following assertions hold true:
\begin{itemize}
\item[a)] Let $\lambda > \lambda _{\infty } :=\lambda _1
/f_p'(\infty)$. Then there exists $R>0$ such that $u\neq
(-\Delta_p)^{-1} (\lambda f^{+}(u) +\tau )$ for every $\tau \geq 0$
and every positive $u\in W_0^{1,p}(\Omega )$ such that $\| u\|
_{W_0^{1,p}} \geq R$.
\item[b)] $i(\Psi _{+}(
\cdot,\lambda ),0)=0$ for all $\lambda > \lambda _{\infty }$.
\end{itemize}
\end{lemma}

\begin{proof}
In order to prove a) we argue by contradiction. Actually, our
argument is similar to the one we used above when proving Lemma
\ref{4.15 AM} part a), but in this case it is more involved because
of the $\tau$-term. Assume there exist $\{ \tau _n \} _n \subset
[0,\infty )$ and a sequence $\{ u_n \} _n \subset W_0^{1,p} (\Omega
)$ of nonnegative functions such that $\| u_n \| _{W_0^{1,p}}
\longrightarrow \infty$ \, as \, $n\rightarrow \infty$ and
\begin{equation}
\label{seq u-n lema 4.16} u_n = (-\Delta_p)^{-1} (\lambda f^{+}(u_n
) +\tau _n ) \quad \text{ for every } n\in \mathbb{N}.
\end{equation}
Since $f^{+}(t)=f_p'(\infty )|t|^{p-2}t+g(t),$ where
$g(t)/|t|^{p-2}t \longrightarrow 0$ as $t\rightarrow +\infty $,
\eqref{seq u-n lema 4.16} can be written as
\begin{equation}
\label{seq u-n lema 4.16 2} u_n = (-\Delta_p)^{-1} (\lambda
f_p'(\infty ) {u }^{p-1} _n+\lambda g(u_n ) +\tau _n ) \quad \text{
for every } n\in \mathbb{N}.
\end{equation}
Let $v_n =u_n /\| u_n \| _{W_0^{1,p}} $ for every  $n\in
\mathbb{N}$. Then $v_n $ satisfies equation
\begin{equation}
\label{seq v-n lema 4.16 } v_n = (-\Delta_p)^{-1} \left( \lambda
f_p'(\infty ) {v }^{p-1} _n+\lambda \frac{g(u_n )}{\| u_n \| ^{p-1}}
+\frac{\tau _n }{\| u_n \| ^{p-1}} \right) \quad \text{ for all }
n\in \mathbb{N}.
\end{equation}
We may assume (by passing to a subsequence) that either
\begin{itemize}
\item[i)] $\frac{\tau _n }{\| u_n \| ^{p-1}} \longrightarrow c\geq
0$ \, as \, $n\rightarrow \infty$, or
\item[ii)] $\frac{\tau _n }{\| u_n \| ^{p-1}} \longrightarrow +\infty$ \, as \, $n\rightarrow \infty$.
\end{itemize}

Let us consider case i). Assume first that $c=0$. Since $\| v_n \|
_{W_0^{1,p}} =1$ for every $n\in \mathbb{N}$, we can suppose (by
taking a subsequence) that there exists $\overline{v}\in
W_0^{1,p}(\Omega )$ such that $v_n \rightharpoonup v$ (weakly) in
$W_0^{1,p}(\Omega )$ and \eqref{seq u-n lema 4.15 } holds true.
Arguing as in the proof of Lemma \ref{4.15 AM},
\begin{equation}
\label{weak conv 1} \lambda f_p'(\infty ) {v  }^{p-1} _n
\rightharpoonup \lambda f_p'(\infty ) {v }^{p-1} \quad \text{ and }
\quad  \frac{g(u_n )}{\| u_n \| ^{p-1}} \rightharpoonup
 0 \quad \text{ weakly in }\quad L^{p'} (\Omega ),
\end{equation}
and by our assumption that $c=0$ in i),
\begin{equation}
\label{weak conv 2} \frac{\tau _n}{\| u_n \| ^{p-1}} \rightharpoonup
 0 \quad \text{ weakly in }\quad L^{p'} (\Omega ).
\end{equation}

Since $(-\Delta_p)^{-1} : L^{p'} (\Omega ) \longrightarrow
W_0^{1,p}({\Omega } )$ is a compact operator, it follows from
\eqref{seq v-n lema 4.16 }, \eqref{weak conv 1} and \eqref{weak conv
2}
\begin{equation}
\label{v positive eigenfunction}
 v = (-\Delta_p)^{-1} \left( \lambda
f_p'(\infty ) {v }^{p-1}  \right) \Leftrightarrow -\Delta_p v =
\lambda f_p'(\infty ) {v }^{p-1}.
\end{equation}
Arguing as we did above after getting \eqref{21-b}, we get that the
nonnegative function $\overline{v}$ is also nonzero. Thus, \eqref{v
positive eigenfunction} provides a contradiction since $\lambda >
\lambda _1 (p)/ f_p'(\infty )$.

\vskip 8pt

We now assume $\frac{\tau _n }{\| u_n \| ^{p-1}} \longrightarrow c>0
$ \, as \, $n\rightarrow \infty$. Let $\varepsilon \in (0, \lambda
f_p'(\infty )-\lambda _1 )$. Using a standard argument, the
following claim can be demonstrated.\\

\emph{\textbf{Claim}}: there exists a large $n$ such that $v_n$ is
 a weak positive supersolution $\omega \in W_0^{1,p}({\Omega } )$ of
problem
\begin{equation}\label{supersolution-omega}
\begin{cases}
\begin{aligned}
 -\Delta _p \omega &= (\lambda _1 +\varepsilon )\omega ^{p-1} \quad \ \text{ in }\, \Omega,\\
 \omega &=0  \quad \ \text{ on }\, \partial \Omega.
\end{aligned}
\end{cases}
\end{equation}

\vskip 8pt

Now, for every $t>0$ and a positive eigenfunction $\phi _1$
corresponding to $\lambda _1$, $t\phi _1$ is a subsolution of
problem \eqref{supersolution-omega}. Let $v_n$ be a positive
supersolution of \eqref{supersolution-omega}. Using that
$\frac{\partial v_n }{\partial \overrightarrow{n}} <0$ and
$\frac{\partial \phi _1 }{\partial \overrightarrow{n}} <0$ on
$\partial \Omega$ (where $\overrightarrow{n}$ denotes the outer unit
normal on $\partial \Omega$), one can prove there exists $t>0$ such
that $t\phi _1 \leq v_n$ on $\Omega$. Using standard truncation and
penalization techniques (see e.g. \cite{DKT}, the appendix in
\cite{GS}, or
 Section 4.5 in \cite{GP}), it can be proved the existence of a solution $\omega
\in W_0^{1,p}({\Omega } ) \cap L^{\infty }(\Omega )$, of problem
\eqref{supersolution-omega}, such that $t\phi _1 \leq \omega \leq
v_n$ in $\Omega$. Thus $\omega$ is a positive eigenfunction
corresponding to the eigenvalue $\lambda _1 +\varepsilon \neq \lambda
_1$. This is a contradiction that shows case
i) above cannot actually occur.\\

Let us now consider case ii). Arguing as in case i), from \eqref{seq
v-n lema 4.16 } it follows that, for  $n\in \mathbb{N}$ sufficiently
large, inequality $-\Delta _p v_n \geq \lambda \gamma {v }^{p-1} _n$
holds true. Then, the same argument as presented in case i) follows,
and we also get a contradiction. We have completed the proof of part
a).

\vskip 12pt

We now prove b). Let  a) $\lambda > \lambda _{\infty }$. From a),
taking $\tau= t$, we know that $u\neq (-\Delta_p)^{-1} (\lambda
f^{+}(u) +t  )$ for every $t\in [0,1]$ and every $u\in
W_0^{1,p}(\Omega )$ with $\| u\| _{W_0^{1,p}} \geq R$. Using again
inversion $z=u/\| u\|^2 _ {{W_0^{1,p}}}$ and the homogeneity of
$(-\Delta_p)^{-1}$, we observe that
\begin{equation}
\label{degree at zero} z\neq (-\Delta_p)^{-1} \left( \lambda \|
z\|^{2(p-1)} f^{+}\left( z/\| z\|^2 \right) +t \right)
\end{equation}
 for every $t\in [0,1]$ and every $z\in
W_0^{1,p}(\Omega )$ such that $0<\| z\| _{W_0^{1,p}} \leq 1/R$. Let
$\varepsilon \in (0, 1/R )$. We now define homotopy $H: [0,1]\times
B_\varepsilon (0)  \longrightarrow W_0^{1,p}(\Omega )$ as
$$H(t,z)=z-(-\Delta_p)^{-1} \left( \lambda \| z\|^{2(p-1)} f^{+}\left( z/\|
z\|^2 \right) +t  \right) \quad \text{for every } \quad z\neq 0, $$
and $H(t,0):= -(-\Delta_p)^{-1} (t)$. Using the same arguments we
used above it can be proved that $H$ is actually continuous, and
also that it is of the form identity $-$ compact.

\vskip 8pt

Using the homotopy invariance property of Leray-Schauder degree, we
obtain
$$\deg(H(0,\cdot), B_{\varepsilon } (0),0)=\deg(H(1,\cdot), B_{\varepsilon } (0),0).$$
On the other hand, $\deg(H(0,\cdot), B_{\varepsilon } (0),0)=\deg(\Psi
_{+}( \cdot,\lambda ), B_{\varepsilon } (0),0)$ and, from  \eqref{degree at
zero} and the definition of $H$,

$$\deg(H(1,\cdot), B_{\varepsilon } (0),0)=0. $$
\end{proof}

\emph{Proof of Theorem \ref{AM-lemma-around-infty}.} Lemmas
\ref{4.15 AM} and \ref{4.16 AM} assert that $i(\Psi _{+}(.,\lambda
),0)=1$ when $\lambda <\lambda _{\infty }$, and $i(\Psi
_{+}(.,\lambda ),0)=0$ when $\lambda >\lambda _{\infty }$. The fact
that these two local degrees are different allows one to repeat the
original arguments used by P. Rabinowitz to prove his global
bifurcation theorem (see \cite{Ra}, \cite{R}, and \cite{AM} Sections
4.3 and 4.4). $\square$

\vskip 8pt

We now prove the existence of two solutions for problem \eqref{PL}.
Since $\Sigma _{\infty }^{+}$ bifurcates from $(\infty ,\lambda _1
/{f_p}'(\infty ) )$, there exist elements $(u,\lambda ) \in \Sigma
_{\infty }^{+}$ such that $\| u \| _{W_0^{1,p} (\Omega )}$ is
arbitrarily large and $\lambda $ is near $\lambda _1 /f_p'(\infty)$.
Hence, because of inequality \eqref{L2-1} in Lemma \ref{L2}, there
exist elements $(u,\lambda ) \in \Sigma _{\infty }^{+}$ such that
$\mathcal{N} _{\infty } (u,\lambda )=\| u \| _{L^{^{\infty}} (\Omega
)} >\alpha$. Lemma \ref{continuity-sup-norm} implies that
$\mathcal{N} _{\infty } (\overline{\Sigma _{\infty }^{+}})$ is
connected. Thus, Lemma \ref{est-sup-norm2} implies that
\begin{equation}
\label{norm-Linfty-in-Sigma-inf-+} \| u \| _{L^{{\infty}} (\Omega )}
>\alpha \quad \forall (u,\lambda )\in \overline{\Sigma _{\infty }^{+}}.
\end{equation}
Because of inequality \eqref{L2-2} in Lemma \ref{L2},
\begin{equation}
\label{norm-L2-in-Sigma-infty-+} \| u \| _{W_0^{1,p} (\Omega )} >
(K_2 )^{-1} \alpha \quad \forall (u,\lambda )\in \overline{\Sigma
_{\infty }^{+}} \cap (W_0^{1,p} (\Omega ) \times [0,2]).
\end{equation}
Now we claim that there exists an element of the form $(u_1,1)\in
\overline{\Sigma _{\infty }^{+}}$. Let us argue by contradiction.
Assume this is not true. Consider the cylinder
$$P=  \{ (u,\lambda ) \in W_0^{1,p} (\Omega ) \times \mathbb{R} \, : \,
\, \lambda \in [0,1], \, \| u \| _{W_0^{1,p}}  \geq (K_2 )^{-1}
\alpha\}.$$ Hypothesis ($f_2$) implies that $\lambda _1
/f_p'(\infty)<1$. Therefore, from Theorem
\ref{AM-lemma-around-infty} it follows that $intP \cap
\overline{\Sigma _{\infty }^{+}} \neq \emptyset$. Also, since
$\Sigma _{\infty }^{+}$ corresponds to the unbounded connected
component $\Gamma _{\infty }^{+}$ of $\Gamma ^{+}$, then
$int(W_0^{1,p} (\Omega ) \times \mathbb{R}\setminus P) \cap
\overline{\Sigma _{\infty}^{+}} \neq \emptyset$. From
\eqref{norm-L2-in-Sigma-infty-+} and our assumption, $\partial P
\cap \overline{\Sigma _{\infty }^{+}}=\emptyset$. Thus, $\partial P$
separates $\overline{\Sigma _{\infty }^{+}}$, i.e.
$$\overline{\Sigma _{\infty }^{+}} \subset int P \cup int(W_0^{1,p} (\Omega )
\times \mathbb{R}\setminus P),$$ which contradicts the connectedness
of $\overline{\Sigma _{\infty }^{+}}$. This contradiction shows
there exists $(u_1,1)\in \overline{\Sigma _{\infty}^{+}}$. From
Theorem \ref{AM-lemma-around-infty}, $u_1 \neq 0$, i.e. $(u_1 ,1
)\in \Sigma _{\infty}^{+} \subset S^{+}$. As mentioned above, this
means $u_1
>0$ on $\Omega$  and $u_1$ satisfies \eqref{PL}. In a similar
fashion we obtain a negative solution $v_1$. The previous argument
shows these two solutions have $L^{\infty}$-norm greater than
$\alpha$.

\subsection{Bifurcation from zero}

First we state the following analogue of Theorem
\ref{AM-lemma-around-infty}.


\begin{theorem}
\label{AM-lemma-around-zero} There exists an unbounded connected
component $\Sigma _0^{+} $ of $S^{+}$ so that $(0,\lambda _1
/f_p'(0) )$ belongs to $\overline{\Sigma _0^{+}}$ and if $(0,\lambda
) \in \overline{\Sigma _0^{+}}$ then $\lambda =\lambda _1 /f_p'(0)$.
Also, there exists an unbounded connected component $\Sigma _0^{-} $
of $S^{-}$ such that $(0,\lambda _1 /f_p'(0) )\in \overline{\Sigma
_0^{-}}$ and if $(0,\lambda ) \in \overline{\Sigma _0^{-}}$ then
$\lambda =\lambda _1 /f_p'(0)$.
\end{theorem}
\textbf{Remark:} This result is essentially an adaptation of Lemma
3.1 in \cite{DM} to our case, and it can be proved either by
following the arguments of \cite{DM} (Theorem 1.1 and Lemma 3.1) or
by using the same ideas we used above to prove Theorem
\ref{AM-lemma-around-infty}.

\vskip 8pt

We now prove the existence of two additional solutions for problem
\eqref{PL}. Since \linebreak $(0,\lambda _1 /{f_p}'(0) )\in
\overline{\Sigma _0^{+}}$, there exist elements $(u,\lambda ) \in
\Sigma _0^{+}$ such that $\| u \| _{W_0^{1,p} (\Omega )}$ is close
to zero and $\lambda $ is near $\lambda _1 /f_p'(0)$. Hence, because
of inequality \eqref{L2-2} in Lemma \ref{L2}, there exist elements
$(u,\lambda ) \in \Sigma _0^{+}$ such that $\mathcal{N} _{\infty }
(u,\lambda )=\| u \| _{L^{^{\infty}} (\Omega )} <\alpha$. From Lemma
\ref{continuity-sup-norm} it follows that $\mathcal{N} _{\infty }
(\overline{\Sigma _0^{+}})$ is connected. Thus, Lemma
\ref{est-sup-norm2} implies that
\begin{equation}
\label{norm-Linfty-in-Sigma-0-+} \| u \| _{L^{{\infty}} (\Omega )}
<\alpha \quad \forall (u,\lambda )\in \overline{\Sigma _0^{+}}.
\end{equation}
Because of inequality \eqref{L2-1} in Lemma \ref{L2},
\begin{equation}
\label{norm-L2-in-Sigma-0-+} \| u \| _{W_0^{1,p} (\Omega )} < K_1
\alpha \quad \forall (u,\lambda )\in \overline{\Sigma _0^{+}} \cap
(W_0^{1,p} (\Omega ) \times [0,2]).
\end{equation}
Now we claim that there exists $(u_2,1)\in \overline{\Sigma
_0^{+}}$. Let us argue by contradiction. Assume this is not true.
Define the cylinder
$$P=  \{ (u,\lambda ) \in W_0^{1,p} (\Omega ) \times \mathbb{R} \, : \,
\, \lambda \in [0,1], \, \| u \| _{W_0^{1,p}}  \leq K_1 \alpha\}.$$
Hypothesis ($f_2$) implies that $\lambda _1 /f_p'(0)<1$. Therefore,
from Theorem \ref{AM-lemma-around-zero} it follows that $intP \cap
\overline{\Sigma _0^{+}} \neq \emptyset$. Also, the unboundedness of
$\overline{\Sigma _0^{+}}$ implies $int(W_0^{1,p} (\Omega ) \times
\mathbb{R}\setminus P) \cap \overline{\Sigma _0^{+}} \neq
\emptyset$. From \eqref{norm-L2-in-Sigma-0-+} and our assumption,
$\partial P \cap \overline{\Sigma _0^{+}}=\emptyset$. Thus,
$\partial P$ separates $\overline{\Sigma _0^{+}}$, i.e.
$$\overline{\Sigma _0^{+}} \subset int P \cup int(W_0^{1,p} (\Omega )
\times \mathbb{R}\setminus P),$$ which contradicts the connectedness
of $\overline{\Sigma _0^{+}}$. This contradiction shows there exists
$(u_2,1)\in \overline{\Sigma _0^{+}}$. From Theorem
\ref{AM-lemma-around-zero}, $u_2 \neq 0$, i.e. $(u_2 ,1 )\in \Sigma
_0^{+} \subset S^{+}$. As mentioned above, this means $u_2
>0$ on $\Omega$ and $u_2$ satisfies \eqref{PL}.

\vskip 8pt

Arguing in a similar fashion with $\overline{\Sigma _0^{-}}$, the
existence of a negative solution $v_2$ of \eqref{PL} is obtained.
From \eqref{norm-Linfty-in-Sigma-0-+} (and its analogue for
$\overline{\Sigma _0^{-}}$) we have $\| u_2 \| _{L^{{\infty}} }, \|
v_2 \| _{L^{{\infty}} } <\alpha$. We summarize the arguments
presented above in the following bifurcation diagram.

\newpsobject{grilla}{psgrid}{subgriddiv=1,griddots=10,gridlabel=6pt}
\begin{center}
\psset{xunit=1cm,yunit=1cm}
\begin{pspicture}(-1,-1)(10.4,10)
 \psaxes[labels=none,ticksize=0.3pt]{<->}(0,0)(-1,9)(9,-1)
\uput[r](-0.9,9.5){$||\cdot||_{W_0^{1,p}}$}
 \uput[r](9.2,0){$\lambda$}

 \psplot[plotstyle=curve,linewidth=10.1pt]{0.7}{7.6}%
{200 x mul sin 0.3 mul 4 add}
 \uput[r](7.8,4.2){$||u||_{L^\infty}=\alpha$}

\uput[r](9.1,2){$\overline{\Sigma_0^{+}}$}
\uput[r](9.1,3){$\overline{\Sigma_0^{-}}$}
\uput[r](9.1,5.8){$\overline{\Sigma_\infty^{+}}$}
\uput[r](9.1,7){$\overline{\Sigma_\infty^{-}}$}
\psplot[plotstyle=curve,linewidth=1.1pt]{1}{9}%
 { x  1 sub sqrt }
   \psplot[plotstyle=curve,linewidth=1.1pt]{1}{9}%
 { x  sqrt 1 sub}
  \uput[r](0.4,-0.4){$\frac{\lambda_1}{f_p^\prime(0)}$}


  \psplot[plotstyle=curve,linewidth=1.1pt]{1.5}{9}%
 { x  1.5 sub sqrt neg 1.2 mul 9 add}
 \psplot[plotstyle=curve,linewidth=1.1pt]{1.5}{9}%
 { x  0.734375 sub sqrt neg 9.875 add}
  \uput[r](.9,9.4){$\frac{\lambda_1}{f_p^\prime(\infty)}$}
 \psline[linestyle=dashed]{<->}(6.5,-0.4)(6.5,9.4)
\uput[r](5.9,-0.5){$\lambda=1$}
\end{pspicture}
\end{center}

\end{document}